\documentclass{amsart}
\usepackage{amsmath, amsthm, amssymb, amsfonts}
\usepackage{bm}
\usepackage{dsfont}
\usepackage[utf8]{inputenc}
\usepackage{rotate}
\usepackage{tikz}
\usepackage{tikz-cd}
\usepackage[arrow,matrix,curve]{xy} 	
\usepackage{xcolor}

\theoremstyle{plain}
\newtheorem{theorem}{Theorem}[section]

\newtheorem{lemma}[theorem]{Lemma}

\theoremstyle{definition}

\newcommand{\norm}[1]{\left\lVert#1\right\rVert}

\title[Multi-dimensional Kronecker Sequences and Gap Lengths]{Multi-dimensional Kronecker Sequences with a Small Number of Gap Lengths}
\date{\today}
\author{Christian Wei\ss{}}

\begin{document}

\maketitle

\begin{abstract} Recently, generalizations of the classical Three Gap Theorem to higher dimensions attracted a lot of attention. In particular, upper bounds for the number of nearest neighbor distances have been established for the Euclidean and the maximum metric. It was proved that a generic multi-dimensional Kronecker attains the maximal possible number of different gap lengths for every sub-exponential subsequence. We mirror this result in dimension $d \in \left\{ 2, 3 \right\}$ by constructing Kronecker sequences which have a surprisingly low number of different nearest neighbor distances for infinitely $N \in \mathbb{N}$. Our proof relies on simple arguments from the theory of continued fractions. %However, it is in general unknown if the known bounds are sharp. We add to the discussion by showing how the second nearest neighbor graph can be used to improve the established upper bounds. Most importantly, we prove that the best possible upper bound in dimension $d \geq 3$ for the Euclidean metric is $\geq 9$ implying that a conjecture of Haynes and Marklof, namely that it is $\leq 7$, does not hold. Finally, we construct special sequences in dimension $d \in \left\{ 2,3\right\}$ which have a surprisingly low number of different nearest neighbor distances for infinitely many $N \in \mathbb{N}$. 
\end{abstract}
\section{Introduction}
It is usually a great challenge to motivate topics from current mathematical research to a wider audience. Happily, this can be relatively easily done for finite gap properties of sequences and holds particularly true in the one-dimensional case: mark the north pole of a circle with red color, let this point rotate by an angle $\alpha \in \mathbb{R}$ and again mark the point, where the north pole lands, red. Let us repeat this procedure $N$ times. Then a remarkable phenomenon occurs. There are always at most three distinct distances between pairs of points in adjacent positions around the circle, compare Figure~\ref{fig:Three_Gap}. This property was conjectured by Steinhaus and first proved by S\'os in \cite{Sos58} and is since then known as Three Gap Theorem. It is also very nicely motivated in \cite{MS17}. In a formal sense, the corresponding sequence is defined as $(\left\{ n\alpha \right\})_{n \in \mathbb{N}}$, where $\left\{ \alpha \right\} := \alpha - \lfloor \alpha \rfloor$ denotes the fractional part of $\alpha$, and called a Kronecker sequence.
 \begin{center}
	\begin{tikzpicture}[scale=0.75]
	\def\ang{137.508};
	\def\scal{2};
	\draw [black,thick,dotted,domain=0:52.524] plot ({\scal*sin(\x)}, {\scal*cos(\x)});
	\draw [black,thick,dotted,domain=52.524:105.048] plot ({\scal*sin(\x)}, {\scal*cos(\x)});
	\draw [blue,thick,domain=105.048:137.508] plot ({\scal*sin(\x)}, {\scal*cos(\x)});
	\draw [black,thick,dotted,domain=137.508:190.032] plot ({\scal*sin(\x)}, {\scal*cos(\x)});
	\draw [green,thick,dashed,domain=190.032:275.016] plot ({\scal*sin(\x)}, {\scal*cos(\x)});
	\draw [black,thick,dotted,domain=275.016:327.54] plot ({\scal*sin(\x)}, {\scal*cos(\x)});
	\draw [blue,thick,domain=327.54:360] plot ({\scal*sin(\x)}, {\scal*cos(\x)});
	\draw [red,fill=red] (\scal*0,\scal*1) circle (.5ex) node[above] {$0$};
	\draw [red,fill=red] ({\scal*sin(52.524)},{\scal*cos(52.524)}) circle (.5ex) node[above] {$3$};
	\draw [red,fill=red] ({\scal*sin(105.048)},{\scal*cos(105.048)}) circle (.5ex) node[right] {$6$};
	\draw [red,fill=red] ({\scal*sin(137.508)},{\scal*cos(137.508)}) circle (.5ex) node[below] {$1$};
	\draw [red,fill=red] ({\scal*sin(190.032)},{\scal*cos(190.032)}) circle (.5ex) node[below] {$4$};
	\draw [red,fill=red] ({\scal*sin(275.016)},{\scal*cos(275.016)}) circle (.5ex) node[left] {$2$};
	\draw [red,fill=red] ({\scal*sin(327.54)},{\scal*cos(327.54)}) circle (.5ex) node[above] {$5$};
	\end{tikzpicture}\\[6pt]
	Figure 1. Illustration of Three Gap Theorem for $N=6$ and rotation by $\alpha = \pi(3-\sqrt{5})$, the golden angle, i.e. $z=(3-\sqrt{5})/2$. \label{fig:Three_Gap}
\end{center}
After the first proof by S\'os, plenty of other proofs of Three Gap Theorem have been found, see e.g. \cite{MS17,Tah17,Wei20}. Besides the nice geometric property of its gap structure, sequences $(\left\{ n\alpha \right\})_{n \in \mathbb{N}}$ are also important for uniform distribution theory because there is a classical class of examples of low-discrepancy sequences among them, see e.g. \cite{Nie92}. Another consequence of the finite gap property is that it prevents the pair correlation statistics of Kronecker sequences from being Poissonian, see \cite{LS20}.\\[12pt]
Three Gap Theorem has been generalized in numerous ways. A comprehensive summary with an extensive list of corresponding literature is given in \cite{HM20} and we refer the interested reader to this paper and references therein for more details. Here, we will concentrate on one specific type of generalization and consider multi-dimensional Kronecker sequences: Let $\mathcal{L}$ be a unimodular lattice in $\mathbb{R}^d$ and consider the $d$-dimensional torus $\mathds{T}_d = \mathbb{R}^d / \mathcal{L}$. For $\alpha \in \mathbb{R}^d$ the $d$-dimensional Kronecker sequence is defined by
$$S_N:=S_N(\alpha,\mathcal{L}):= \left\{ (z_n)_{n \in \mathbb{N}} := n\alpha + \mathcal{L} \right\} \subset \mathds{T}_d.$$
If $\alpha \in \mathbb{Q}^d$, then the sequence $(z_n)_{n \in \mathbb{N}}$ is periodic and thus the number of distances between elements in $S_N(\alpha,\mathcal{L})$ is universally bounded. Let $\norm{\cdot}_q$ denote a $L_q$-norm on the $d$-dimensional torus, where $1 \leq q \leq \infty$ and let their implied metrics be $d_q(\cdot,\cdot)$. The distance $\delta_{n,N}^1:=d_q(z_n,nn_1(z_n))$ is the distance of $z_n \in S_N$ to its closest neighbor $nn_1(z_n) \in S_N \setminus \{ z_n \}$ in terms of the $L_q$-metric. Finally, the number $g_N(\alpha,\mathcal{L},\norm{\cdot})$ denotes the number of distinct nearest neighbors of the finite sequence $S_N$ with respect to the norm $\norm{\cdot}$. For the case of the Euclidean metric ($q=2$), the best known bound is due to \cite{HM20} improving results from \cite{BS08}: for any $\alpha, \mathcal{L}$ and $N$ the inequality
\begin{align} \label{ineq:gn2}
g_N(\alpha,\mathcal{L},\norm{\cdot}_2) \leq \begin{cases} 3 & \text{if} \ d = 1 \\ 5 & \text{if} \ d = 2\\ \sigma_d +1 & \textit{if} \ d \geq 3\end{cases}
\end{align}
holds, where $\sigma_d$ is the kissing number, i.e. the maximum number of non-overlapping spheres of radius one in $\mathbb{R}^d$ which can be arranged such that they touch the unit sphere in exactly one point. Let us denote by $g(d,2)$ the maximal possible value of any $g_N(\alpha,\mathcal{L},\norm{\cdot}_2)$ in dimension $d$. The bounds for dimension $d \in \left\{ 1,2\right\}$ in \eqref{ineq:gn2} are known to be sharp. Moreover, in dimension $d=2$, the upper bound can also be calculated for the more general case that only neighbors with distances in directions $\mathcal{D}$, where $\mathcal{D} \subset \mathbb{S}^1_1$ is an half-open interval, are considered. The upper bound for $g_N(\mathcal{D},\alpha,\mathcal{L},\norm{\cdot})$ then depends on the arclength of $\mathcal{D}$, see \cite{HM20}, Theorem 4.\\[12pt]
In practice, it turns out surprisingly hard even in dimension $2$ to find explicit examples of combinations $\alpha,\mathcal{L},N$ with $g_N(\alpha,\mathcal{L},\norm{\cdot}_2) = g(2,2) = 5$. Nonetheless, it is theoretically known from \cite{HM20} that reaching the (unknown) upper bound $g(d,2)$ is not at all a singular event bur rather the general case.
\begin{theorem}[Haynes, Marklof \cite{HM20}, Theorem 2] \label{thm:hm:lower_bound} Let $\mathcal{L}, \mathcal{L}_0$ be unimodular lattices. There is $P \subset \mathbb{R}^d$ of full Lebesgue measure, such that for every $\alpha \in P, a_0 \in \mathbb{R}^d$ and for every sub-exponential sequence $(N_i)_{i \in \mathbb{N}}$ we have
$$\limsup_{i \to \infty} g_{N_i}(\alpha,\mathcal{L}) \geq \sup_{N \in \mathbb{N}} g_N(\alpha_0,\mathcal{L}_0).$$
\end{theorem}
In particular, the maximal possible value $g(d,2)$ can be realized for an infinite sequence of $(N_{i_j})_{j \in \mathbb{N}}$ although the theorem does not give an explicit way to find the set $P$. Moreover, Haynes and Marklof conjectured that in dimension $d=3$ even $g_N(3,2) \leq 7$ holds and gave an explicit example with $g_N(\alpha,\mathcal{L},\norm{\cdot}_2) = 7$. Similar results were obtained in \cite{HR20} for the $L_\infty$-norm\\[12pt]
While Theorem~\ref{thm:hm:lower_bound} describes the generic case, we construct in this paper sequences in dimension $d \in \left\{2,3\right\}$ which have a surprisingly small number of nearest neighbor distances, namely satisfying $g_N(\alpha,\mathcal{L},\norm{\cdot}_q) = 1$ for infinitely many $N \in \mathbb{N}$ and \textit{all} $1 \leq q \leq \infty$ simultaneously. The construction relies on the continued fraction expansion of $\alpha_1,\alpha_2$ and $\alpha_3$ and makes sure that the denominators $q_{i,j}$ of the convergents of all $\alpha_j$ are equal for infinitely $i \in \mathbb{N}$. By that we obtain our theorem. For the sake of simplicity and readability we omit here the exact number theoretical properties (because they are in our eyes not relevant at this stage) and move the details to Lemma~\ref{lem:cond}.
\begin{theorem} \label{thm:construction:3d} Let $[0;a_1^1,a_2^1,\ldots]$ be the continued fraction of $\alpha_1 \in (0,1)$ and let $(q_k^1)_{k \in \mathbb{N}}$ denote the denominators of the corresponding convergents. Assume that all $a_i^1 > 1$ for all $i \in \mathbb{N}$ and that $(q_k^1)_{i \in \mathbb{N}}$ has a subsequence $(q_{k_l})_{l \in \mathbb{N}}$ which satisfies the conditions of Lemma~\ref{lem:cond} for all $l \in \mathbb{N}$. Then there exists an $\alpha_2,\alpha_3$ such that  $(\alpha_1,\alpha_2,\alpha_3)$ satisfies $g_{N}(\alpha,\mathcal{L},\norm{\cdot}_q) = 1$ for infinitely $N \in \mathbb{N}$ and all $1 \leq q \leq \infty$ simultaneously.
\end{theorem}

\section{Sequences with few nearest neighbor distances}
We will now show that there exist vectors $(\alpha_1,\alpha_2,\alpha_3)$ such that $g_N(\alpha,\mathcal{L},\norm{\cdot}_q)=1$ for infinitely many $N \in \mathbb{N}$ independent of $1 \leq q \leq \infty$. The construction of the $\alpha_i$ is based on properties of the continued fraction expansion. Therefore, we briefly fix notation and summarize some of their important properties. For more details, we refer the reader to \cite{Bs96,Nie92}. Let $[a_0^i;a_1^i,\ldots]$ be the continued fraction expansion of $\alpha_i$ and denote the corresponding sequence of convergents by $(p_n^i/q_n^i)_{n \in \mathbb{N}_0}$. Recall that
\begin{align} \label{eq:p}
p_{-2} = 0, p_{-1} = 1, p_n = a_np_{n-1} + p_{n-1}, n \geq 0
\end{align}
\begin{align} \label{eq:q}
q_{-2} = 1, q_{-1} = 0, q_n = a_nq_{n-1} + q_{n-1}, n \geq 0
\end{align}
 Since every unimodular lattice $\mathcal{L} \subset \mathbb{R}^d$ is of the form $\mathbb{Z}^dM$ with $M \in \text{SL}(d,\mathbb{R})$, we may in the following restrict ourselves to the case $\mathcal{L} = \mathbb{Z}^d$. Let $V = \left\{ v_1, v_2, \ldots, v_N \right\}$ be a set of points in $\mathbb{R}^d$ and let $d(x,y)$ be a metric on $\mathbb{R}^d$. A \textit{nearest neighbor} of $v_i$ is a point $v_j$ with minimum distance from $v_i$. In order to make $v_j$ unique, we use an idea which we found in \cite{EPY97}, and let $v_j$ be the maximum index in $V$ with this property and denote it by $nn_1(v_i)$. This definition is slightly different than the one used to draw the figures in \cite{HM20}, where any point having the same distance from $v_i$ as $nn_1(v_i)$ is a nearest neighbor. For any $v_i$ we define an edge by $e_1(v_i) = \langle v_i, nn_1(v_i) \rangle$ and obtain the \textit{nearest neighbor graph} $(V,E_1)$ where $E_1=\left\{ e_1(v_i) | v_i \in V \right\}$.\\[12pt] 
 \paragraph{The Counting Metric.} Besides the $L_q$ metric on $\mathbb{R}^d$ we also need to consider the counting metric of the sequence. For $i\alpha \in S_N$ let $h_i(N) := \left| \frac{nn_1(i\alpha)}{\alpha} -i\right|$ be the closest neighbor in the counting metric of the sequence. Note that $h_i(N) > h_i(M)$ implies that $h_1(N-i+1) > h_1(M-i+1)$ because $h_1(N)$ is necessarily the first of the $h_i(N)$ to increase. This means that $h_1(N)$ has a special role. Although the proof is elementary, the following lemma plays a crucial role in the remainder of this paper.
\begin{lemma} \label{lem:hi:part1} 
	For all $n,k \in \mathbb{N}$ with $k$ such that $h_1(n+k) = h_1(n)$ we have
	$$h_{1+k}(n+k) = h_{1}(n)$$
\end{lemma}
\begin{proof} %We have $h_1(n) = h_{1+h_1(n))}$ by the definition of the counting metric. Since $h_1(n)$ is the first of the $h_i$ to increase.
	The nearest neighbor of any $\left\{(1+k)\alpha\right\}$ can be pulled back to the nearest neighbor of $\left\{ 1 \alpha \right\}$ by $R_{\alpha}^{-k}$, where $R_{\alpha}(z)= z\alpha$ for $z \in \mathds{T}^d$. Since $h_1$ is the first of the $h_i$ to increase, the claim follows.
\end{proof}
Hence, $h_1(N)$ essentially determines the behavior of the nearest neighbor distance structure. Another simple observation serves as our base for constructing sequences with as few nearest neighbor distances as possible.
\begin{lemma} \label{lem:hi:part2} If there exists a multiindex $(n_1,n_2,\ldots,n_d) \in \mathbb{N}^d$ with $q_{n_1}^1 = q_{n_2}^2 = \ldots = q_{n_d}^d =:q$, then $h_1(q+1) = q$ independent of the $L_p$ metric.
\end{lemma}
\begin{proof} Let $\norm{\cdot}$ be the norm of the one-dimensional torus, i.e. $\norm{x} = \min(x-\lfloor x \rfloor, 1-(x-\lfloor x \rfloor))$, where $\lfloor x \rfloor$ denotes the Gau\ss{} bracket. The claim follows from the fact $\norm{\left\{q_{n_i}^i\alpha_i\right\}} < \norm{\left\{q_{n_i-1}^i\alpha_o\right\}}$ for all $i$ but $\norm{\left\{k \alpha_i\right\}} \geq  \norm{\left\{q_{n_i}^i\alpha_i\right\}}$ for all $k < q_{n_i}^i$.\\[12pt]
\end{proof}
Although Lemma~\ref{lem:hi:part2} (in combination with Lemma~\ref{lem:hi:part1}) is the best achievable result in the general situation, we can prove stronger distance properties for non-generic $\alpha = (\alpha_1,\alpha_2,\alpha_3)$. At first, we present a very simple construction of $\alpha = (\alpha_1,\alpha_2)$, which captures the gist of our idea, before we come to a more general one, which allows for an application in dimension $d=3$, too.\\[12pt]
\paragraph{A simple construction.} Let $\alpha_1 = [a_0^1=0,a_1^1=1,a_2^1,\ldots] \in (1/2,1)$ be arbitrary and choose $\alpha_2$ as the real number which has the following continued fraction expansion $a_0^2=0,a_2^2 = 1 + a_2^1$ and $a_i^2 = a_{i+1}^1$. Consequently by \eqref{eq:q}, we have
\begin{align*}
    q_1^{2} & = 1 = q_2^1,\\
    q_2^{2} & = 1+a_2^1 = q_3^1,\\
    q_i^{2} & = a_i^2q_{i-1}^2 + q_{i-1}^2 = a_{i+1}^1q_{i}^1 + q_{i-1}^1 = q_{i+1}^1 \quad \text{for} \ i \geq 3.
\end{align*}
The construction is equivalent to setting $\alpha_2 = 1 - \alpha_1$. \begin{theorem} \label{thm:1}
    For $\alpha = (\alpha_1,1-\alpha_1)$ the two-dimensional Kronecker sequence has at most three different nearest neighbor distances, i.e. $1 \leq g_N \leq 3$ for all $N \in \mathbb{N}$. If we have in addition $a_1^i > 1$ for all $i$, then there are at most two different nearest neighbor distances, and there is a sequence $(N_{1_i}) \subset \mathbb{N}$ with $g_{N_{1_i}} = 1$ for all $i \in \mathbb{N}$ and a sequence $(N_{2_j}) \subset \mathbb{N}$ with $g_{N_{2_j}} = 2$ for all $j \in \mathbb{N}$.
\end{theorem}
\begin{proof} By Lemma~\ref{lem:hi:part2}, the function $h_1(n)$ increases by $1$ for all $n=q_{i+1}^1+1=q_{i}^2+1$. For $n \in \left\{ q_i^{2}+2, q_i^2+3, \ldots, q_{i+1}^2 \right\}$ the function $h_1(n)$ remains constant because for each vector component we have $\norm{\left\{1\alpha_i\right\} - \left\{k\alpha_i\right\}} > \norm{\left\{1\alpha_i\right\} - \left\{(q_{i}^2+1)\alpha_i\right\}}.$
Since $q_{i+1} \geq q_{i} + q_{i-1}$ there are thus at most three different nearest neighbor distances by Lemma~\ref{lem:hi:part1}. If all the $a_i > 1$, we get again by Lemma~\ref{lem:hi:part1} that $h_i(N_1) = h_1(N_1)$ for all $i =1,2,\ldots,N_1$ if $N_1 \in \left\{ 2q_i^{2} + 1, 2q_i^{2} + 2, \ldots, q_{i+1}^2 \right\}$. Furthermore, it follows that $g_{N_2} = 2$  if $N_2 \in \left\{ q_i^{2} + 1, q_i^{2} + 2, \ldots, 2q_i^{2} \right\}$.
\end{proof}
\paragraph{A more general construction.} In order to find further examples where the one distance property appears, it suffices if the weaker condition is satisfied that the convergents of $\alpha_2$ are a subset of those of $\alpha_1$, i.e. $\left\{(q_i^2) | i \in \mathbb{N}_0\right\} \subset \left\{(q_i^1) | i \in \mathbb{N}_0\right\}$. This cannot only be achieved for $\alpha_2 = 1 - \alpha_1$ but for a much broader class of examples as the following construction will show. The proof of Theorem~\ref{thm:1} can then be used to prove that the  one distance property holds.\\[12pt]
The coefficients of the continued fraction expansion of $\alpha_1$ and $\alpha_2$ have now to be considered simultaneously. We choose $a_1^1,\ldots,a_{k_1}^1$ arbitrarily. This implies that the construction works for a dense subset of $(0,1)$. Then we set $a_2^0 = 0$ and $a_2^1 = q_{k_1}^1$. Hence $q_1^2 = q_{k_1}^1$. 
We proceed inductively and assume that we have already reached $q_l^2 = q_{k_l}^1$. Again we may choose $a_{k_l+1}^1,\ldots,a_{k_{l+1}-2}^1$ arbitrarily. Next we assume that $a_{k_{l+1}-1}^1$ is such that $\gcd(q_{k_{l+1}-1}^1,q_{k_l}^1)=1$. Then $a_{k_{l+1}}^1$ may be any natural number with
\begin{align} \label{eq1}
a_{k_{l+1}}^1q_{k_{l+1}-1}^1+q_{k_{l+1}-2}^1 \mod (q_{k_l}^1)) = q_{l-1}^2.
\end{align}
Finally, we set $a_{l+1}^2 = (a_{k_{l+1}}^1q_{k_{l+1}-1}^1+q_{k_{l+1}-2}^1-q_{l-1}^2)/q_{k_l}^1 \in \mathbb{N}$. Therefore
$$q_{l+1}^2= a_{l+1}^2 q_{l}^2 + q_{l-1}^2 = a_{k_{l+1}}^1q_{k_{l+1}-1}^1 + q_{k_{l+1}-2}^1 = q_{k_{l+1}}^1.$$
The condition $\gcd(q_{k_{l+1}-1}^1,q_{k_l}^1)=1$ holds automatically if $k_{l+1} = k_l + 2$. In all other cases this can be assured if $a_{k_{l+1}}-1$ satisfies two extra conditions. 
\begin{lemma} \label{lem:cond}
Let $\gcd(q_{k_l}^1,q_{k_{l+1}-2}^1) = b$ and define $c:= q_{k_l}^1/b, d:= \gcd(c,q_{k_{l+1}-3}^1)$ and $e := c/d$. Then $\gcd(q_{k_{l+1}-1}^1,q_{k_l}^1)=1$ holds if $\gcd(d,a_{k_{l+1}-1}^1) = 1$ and $e|a_{k_{l+1}-1}^1$.
\end{lemma}
We remark that $c=1$ if $k_{l+1} = k_l+2$ and that $d,e=1$ if $c=1$.
\begin{proof} We consider the prime divisors $p \in \mathbb{N}$ of $q_{k_l}^1$ and show that none of them can divide $q_{k_{l+1}-1}^1$. If $p | b$, then $p|q_{k_{l+1}-2}^1$ and thus $p \nmid q_{k_{l+1}-1}^1$ because we have $\gcd(q_{k_{l+1}-2}^1,q_{k_{l+1}-1}^1)=1$. Otherwise we have that $p | c$ by the definition of $c$ and thus either $p | d$ or $p | e$. If $p | d$, then $p | q_{k_{l+1}-3}$ and therefore $p \nmid q_{k_{l+1}-2}$ and also $p \nmid q_{k_{l+1}-1} = a_{k_{l+1}-1}^1q_{k_{l+1}-2}+q_{k_{l+1}-3}$. If $p | e$, then $p \nmid q_{k_{l+1}-3}$ and hence $p \nmid q_{k_{l+1}-1}$ because $p|a_{k_{l+1}-1}^1$.
\end{proof}
From this we can derive the desired theorem in dimension $d=2$.
\begin{theorem} Let $[0;a_1^1,a_2^1,\ldots]$ be the continued fraction of $\alpha_1 \in (0,1)$ and let $(q_k^1)_{k \in \mathbb{N}}$ denote the denominators of the corresponding convergents. Assume that all $a_i^1 > 1$ for all $i \in \mathbb{N}$ and that $(q_k^1)_{i \in \mathbb{N}}$ has a subsequence $(q_{k_l})_{l \in \mathbb{N}}$ which satisfies the conditions of Lemma~\ref{lem:cond} for all $l \in \mathbb{N}$. Then there exists an $\alpha_2$ whose convergents have denominators $(q_{k_l})_{l \in \mathbb{N}}$ and $\alpha = (\alpha_1,\alpha_2)$ satisfies $g_{N_{1_i}}(\alpha,\mathcal{L},\norm{\cdot}_q) = 1$ for infinitely $N_{1_i} \in \mathbb{N}$ and all $1 \leq q \leq \infty$.
\end{theorem}
\begin{proof} The number $\alpha_2$ exists by the general construction. Additionally we have $h_1(q_{k_l}+1) = q_{k_l}$ for all $l \in \mathbb{N}$ by Lemma~\ref{lem:hi:part2}. As the $a_i^1 > 1$ also the $a_i^2 > 1$ by construction. Hence there are infinitely many $N_{1_i} \in \mathbb{N}$ with $g_{N_{1_i}} = 1$ by Lemma~\ref{lem:hi:part2}.
\end{proof}
Combining the simple construction and the more general approach we also obtain Theorem~\ref{thm:construction:3d}, where we choose $\alpha_3 = 1 - \alpha_1$. The following lemma allows to further generalize this construction. It gives the necessary information about which points of the Kronecker lie closest to the origin.
\begin{lemma} \label{lem:asmallest} Let $\alpha \in \mathbb{R}\setminus \mathbb{Q}$ have continued fraction expansion $[0,a_1,a_2,\ldots]$ and convergents $(p_i,q_i)_{i \in \mathbb{N}}$. Now consider the subset $\left\{ n \alpha \right\}_{n=q_i}^{q_{i+1}}$ of the Kronecker sequence. Then the $2a_{i+1}$ points closest to the origin are given by the following inequality
\begin{align*} 
\norm{\left\{ q_{i+1} \alpha \right\}} < \norm{\left\{ 1 \cdot q_{i} \alpha \right\}} & < \norm{\left\{ (q_{i+1} - 1 \cdot q_{i}) \alpha \right\}} < \norm{\left\{ 2 \cdot q_{i} \alpha \right\}}\\ & < \ldots < \norm{\left\{ (q_{i+1} - a_{i+1} q_{i}) \alpha \right\}} < \norm{\left\{ (a_{1+1} \cdot q_{i}) \alpha \right\}},
\end{align*}
where $\norm{\cdot}$ as usual denotes the one-dimensional torus norm.
\end{lemma}
\begin{proof} This is the dynamics behind Three Gap Theorem, compare e.g.~\cite{Wei20}. \end{proof}
Hence, the right hand side of \eqref{eq1} may be replaced by the condition $b q_{l-1}^2$ with $b \leq a_{k_{l+1}}^1$ and we still get a sequence with only one nearest neighbor distance for infinitely many $N_{1_i}$ by Lemma~\ref{lem:asmallest} if in addition $a_{k_{l+1}}^1 /2b > 2$ is satisfied.
\section*{Acknowledgments}
Parts of the research on this paper was conducted during a stay of the author at the Max-Planck Institute in Bonn whom he would like to thank for hospitality and an inspiring scientific atmosphere.
\bibliographystyle{alpha}
\bibdata{references}
\bibliography{references}
\end{document}